\documentclass[epsfig,latexsym,amsfonts,twoside]{article}
\usepackage{amssymb,enumerate}
\usepackage{amsmath,amscd}

\usepackage{tikz}
\tikzset{help lines/.style={color=blue!50,very thin}}

\def\part#1{\frac{\partial\phantom{#1}}{\partial#1}}
\newtheorem{thm}{Theorem}
\newtheorem{theorem}[thm]{Theorem}
\newtheorem{proposition}[thm]{Proposition}
\newtheorem{lemma}[thm]{Lemma}
\newtheorem{corollary}[thm]{Corollary}
\newtheorem{conjecture}[thm]{Conjecture}

\newenvironment{proof}{\begin{trivlist}\item[]{\bf Proof} }%
{\hfill $\Box$ \end{trivlist}}
{\end{trivlist}}
\newenvironment{remark}{\begin{trivlist}\item[]{\bf Remark} }%
{\end{trivlist}}
{\end{trivlist}}
\newenvironment{question}{\begin{trivlist}\item[]{\bf Question} }%
{\end{trivlist}}

\def\Z{\ifmmode{{\mathbb Z}}\else{${\mathbb Z}$}\fi}
\def\Q{\ifmmode{{\mathbb Q}}\else{${\mathbb Q}$}\fi}
\def\C{\ifmmode{{\mathbb C}}\else{${\mathbb C}$}\fi}
\def\P{\ifmmode{{\mathbb P}}\else{${\mathbb P}$}\fi}

\def\H{\ifmmode{{\mathrm H}}\else{${\mathrm H}$}\fi}

\def\B{\ifmmode{{\mathcal B}}\else{${\mathcal B}$}\fi}
\def\E{\ifmmode{{\mathcal E}}\else{${\mathcal E}$}\fi}
\def\F{\ifmmode{{\mathcal F}}\else{${\mathcal F}$}\fi}
\def\K{\ifmmode{{\mathcal K}}\else{${\mathcal K}$}\fi}
\def\L{\ifmmode{{\mathcal L}}\else{${\mathcal L}$}\fi}
\def\M{\ifmmode{{\mathcal M}}\else{${\mathcal M}$}\fi}
\def\N{\ifmmode{{\mathcal N}}\else{${\mathcal N}$}\fi}
\def\O{\ifmmode{{\mathcal O}}\else{${\mathcal O}$}\fi}
\def\U{\ifmmode{{\mathcal U}}\else{${\mathcal U}$}\fi}
\def\V{\ifmmode{{\mathcal V}}\else{${\mathcal V}$}\fi}
\def\X{\ifmmode{{\mathcal X}}\else{${\mathcal X}$}\fi}

\def\Br{\ifmmode{{\mathrm{Br}}}\else{${\mathrm{Br}}$}\fi}
\def\OG{\ifmmode{\widetilde{\cal M}_4}\else{$\widetilde{\cal M}_4$}\fi}
\def\D{\ifmmode{{\mathcal D}^b}\else{${{\mathcal
    D}^b}$}\fi}
\def\Shah{\ifmmode{\amalg\hspace*{-3.5pt}\amalg}\else{$\amalg\hspace*{-3.5pt}\amalg$}\fi}

\begin{document}

\title{Singular fibres of very general Lagrangian fibrations\footnote{2010 {\em Mathematics Subject
  Classification.\/} 14D06, 14K10, 53C26.}}
\author{Justin Sawon}
\date{May, 2019}
\maketitle

\begin{abstract}
Let $\pi:X\rightarrow\P^n$ be a (holomorphic) Lagrangian fibration that is very general in the moduli space of Lagrangian fibrations. We conjecture that the singular fibres in codimension one must be semistable degenerations of abelian varieties. We prove a partial result towards this conjecture, and describe an example that provides further evidence.
\end{abstract}

\maketitle

\section{Introduction}

In~\cite{sawon16} the author proved that there are finitely many deformation classes of (holomorphic) Lagrangian fibrations $\pi:X\rightarrow\P^n$ satisfying certain natural conditions. One of these conditions was that the singular fibres in codimension one should be semistable degenerations of abelian varieties. We will describe explicitly what this means shortly. At this stage, let us just point out that semistable degenerations are `nice' in the following sense: if $\bar{\mathcal{A}}_n$ denotes a toroidal compactification of the moduli space $\mathcal{A}_n$ of abelian varieties, then the boundary of $\bar{\mathcal{A}}_n$ parametrizes semistable degenerations.

The singular fibres of Lagrangian fibrations in codimension one were studied by Matushita~\cite{matsushita01, matsushita07} and Hwang and Oguiso~\cite{ho09, ho11}. Their classifications include many fibres that are not semistable, because they consider all Lagrangian fibrations, not just very general ones. Matsushita~\cite{matsushita05, matsushita16} proved that inside the moduli space $\mathcal{M}$ of deformations of $X$ as a complex manifold, there is a hypersurface $\mathcal{H}$ parametrizing Lagrangian fibrations; by {\em very general\/} we mean that $X$ corresponds to a very general point of $\mathcal{H}$, i.e., contained in the complement of countably many Zariski closed subsets. Any Lagrangian fibration can be deformed slightly so that it becomes very general.

Hwang and Oguiso's results imply that if a singular fibre $X_t$ {\em in codimension one\/} is semistable then it must be a rank-one semistable degeneration. Explicitly, this means that the normalization of $X_t$ must be a $\P^1$-bundle, or disjoint union of $\P^1$-bundles, over an $(n-1)$-dimensional abelian variety $A$. Write this as
$$\tilde{X}_t=\coprod_{i=1}^k Y_i.$$
Moreover, each $Y_i$ must be a $\P^1$-bundle over $A$ given by projectivizing a topologically trivial rank-two bundle,
$$Y_i\cong\P(\O_A\oplus\mathcal{L})\rightarrow A$$
where $\mathcal{L}\in\mathrm{Pic}^0A$. This implies that $Y_i$ has distinguished zero and infinity sections, $(Y_i)_0$ and $(Y_i)_{\infty}$. The fibre $X_t$ itself is obtained by gluing zero sections to infinity sections, so that 
$$(Y_i)_{\infty}\cong A\qquad\qquad\mbox{is identified with}\qquad\qquad (Y_{i+1})_0\cong A$$
(with $Y_{k+1}$ denoting $Y_1$). In general these identifications $A\cong (Y_i)_{\infty}\rightarrow (Y_{i+1})_0\cong A$ do {\em not\/} take $0$ to $0$, so $X_t$ itself is {\em not\/} a fibration over $A$. Figures $1$ and $2$ show examples with $k=1$ and $k=3$, respectively.

\begin{center}
\vspace*{3mm}
\begin{tikzpicture}[scale=1.0,>=stealth]
  \draw (-2,0) -- (-2,2.5);
  \draw (-6,0) -- (-6,2.5);
  \draw[thick, blue] (-2,0) -- (-6,0);
  \draw[thick, blue] (-2,2.5) -- (-6,2.5);
  \draw (-2,-1.5) -- (-6,-1.5);
  \draw[very thick, red] (-5,0) -- (-5,2.5);
  \draw[very thin, gray, ->] (-4.9,2.4) -- (-4.1,0.1);
  \draw[very thin, gray, ->] (-3.9,2.4) -- (-3.1,0.1);
  \draw[->] (-4,-0.5) -- (-4,-1);
  
  \draw[->] (-1.2,1.2) -- (-0.7,1.2);
  
  \draw (0,0) .. controls (0.2,0.1) and (0.7,0.5) .. (0.8,0.6);
  \draw[dashed] (0.8,0.6) .. controls (1.5,1.3) and (1.5,2.6) .. (0.8,2.7);
  \draw (0.8,2.7) .. controls (0.1,2.5) and (0.1,1.2) .. (0.8,0.6);
  \draw[dashed] (0.8,0.6) .. controls (0.7,0.7) and (1.2,0.3) .. (1.4,0.2);
  
  \draw (4,-1) .. controls (4.2,-0.9) and (4.7,-0.5) .. (4.8,-0.4);
  \draw (4.8,-0.4) .. controls (5.5,0.3) and (5.5,1.6) .. (4.8,1.7);
  \draw (4.8,1.7) .. controls (4.1,1.5) and (4.1,0.2) .. (4.8,-0.4);
  \draw (4.8,-0.4) .. controls (4.7,-0.3) and (5.2,-0.7) .. (5.4,-0.8);
  
  \draw (0,0) -- (4,-1);
  \draw (0.8,2.7) -- (4.8,1.7);
  \draw[thick, blue] (0.8,0.6) -- (4.8,-0.4);
  \draw[dashed] (1.4,0.2) -- (4.6,-0.6);
  \draw (4.6,-0.6) -- (5.4,-0.8);

  \draw[very thick, red] (0.5,-0.125) .. controls (1,0) and (1.2,0) .. (1.8,0.35);
  \draw[dashed, very thick, red] (1.8,0.35) .. controls (2.4,0.7) and (3.3,2.0) .. (2.8,2.2);
  \draw[very thick, red] (2.8,2.2) .. controls (2.0,1.9) and (3.2,0.15) .. (3.8,-0.15);
  \draw[dashed, very thick, red] (3.8,-0.15) .. controls (4.3,-0.45) and (4.3,-0.5) .. (4.7,-0.6);
  
  
  \draw (-6.3,2.5) node {$s_{\infty}$};
  \draw (-6.3,0) node {$s_0$};
  \draw (-6.3,1.2) node {$\mathbb{P}^1$};
  \draw (-6.3,-1.5) node {$A$};
\end{tikzpicture}\\
\vspace*{5mm}
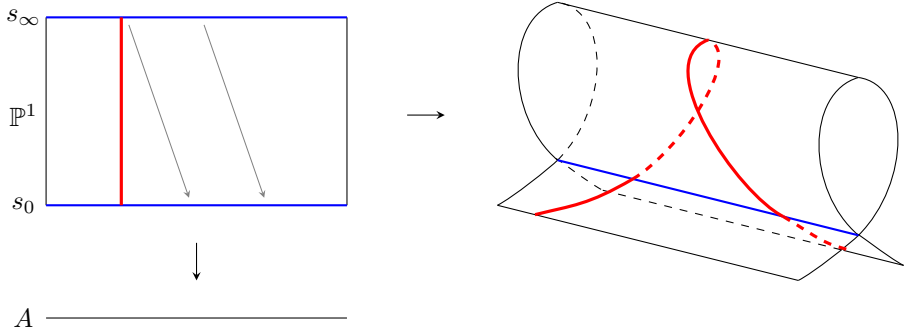
Figure 1: Rank-one semistable degeneration
\end{center}

\begin{center}
\vspace*{3mm}
\begin{tikzpicture}[scale=1.1,>=stealth]
  \draw (-2,-1.5) -- (-2,4.5);
  \draw (-5.5,-1.5) -- (-5.5,4.5);
  \draw[very thick, red] (-4.7,-1.5) -- (-4.7,4.5);
  \draw[thick, blue] (-2,-1.5) -- (-5.5,-1.5);
  \draw[thick, blue] (-2,0.5) -- (-5.5,0.5);
  \draw[thick, blue] (-2,2.5) -- (-5.5,2.5);
  \draw[thick, blue] (-2,4.5) -- (-5.5,4.5);
  \draw[very thin, gray, ->] (-4.6,4.4) -- (-3.8,-1.4);
  \draw[very thin, gray, ->] (-3.6,4.4) -- (-2.8,-1.4);

  \draw[->] (-1.2,1.2) -- (-0.7,1.2);
 
  \draw (0,0.5) -- (0.1,0.76);
  \draw[dashed] (0.1,0.76) -- (0.2,1.02);
  \draw (0.2,1.02) -- (0.8,2.58);
  \draw[dashed] (0.8,2.58) -- (0.9,2.84);
  \draw (0.9,2.84) -- (1,3.1); 
  \draw (1,3.1) -- (5,2.1) -- (4,-0.5) -- (0,0.5);
 
  \draw (5.7,0) -- (4.9,0.2);
  \draw[dashed] (4.9,0.2) -- (1.7,1) -- (0.8,2.5);
  \draw (0.8,2.5) -- (0.5,3);
  \draw (0.5,3) -- (4.5,2) -- (5.7,0);
  
  \draw (-0.3,0.8) -- (0.225,0.975);
  \draw[dashed] (0.225,0.975) -- (1.8,1.5) -- (5.4,0.6);
  \draw (5.4,0.6) -- (5.8,0.5);
   
  \draw (5.8,0.5) -- (3.7,-0.2) -- (-0.3,0.8);
 
  \draw[thick, blue] (0.2,0.95) -- (4.2,-0.05);
  \draw[thick, blue] (0.8,2.5) -- (4.8,1.5);
  \draw[dashed, thick, blue] (1.45,1.4) -- (4.45,0.65);
  \draw[thick, blue] (4.45,0.65) -- (5.45,0.4);

  \draw[very thick, red] (0.7,0.3) -- (0.79,0.45);
  \draw[dashed, very thick, red] (0.79,0.45) -- (0.97,0.75);
  \draw[very thick, red] (0.97,0.75) -- (1.84,2.20);
  \draw[dashed, very thick, red] (1.84,2.20) -- (2.05,2.55);
  \draw[very thick, red] (2.05,2.55) -- (2.2,2.8);

  \draw[very thick, red] (1.48,2.75) -- (1.90,2.22);
  \draw[dashed, very thick, red] (1.90,2.22) -- (3.16,0.63);

  \draw[dashed, very thick, red] (3.05,1.2) -- (2.07,0.57);
  \draw[very thick, red] (2.07,0.57) -- (1.65,0.3);


  \draw (-5.8,-0.5) node {$\mathbb{P}^1$};
  \draw (-5.8,1.5) node {$\mathbb{P}^1$};
  \draw (-5.8,3.5) node {$\mathbb{P}^1$};

\end{tikzpicture}\\
\vspace*{5mm}
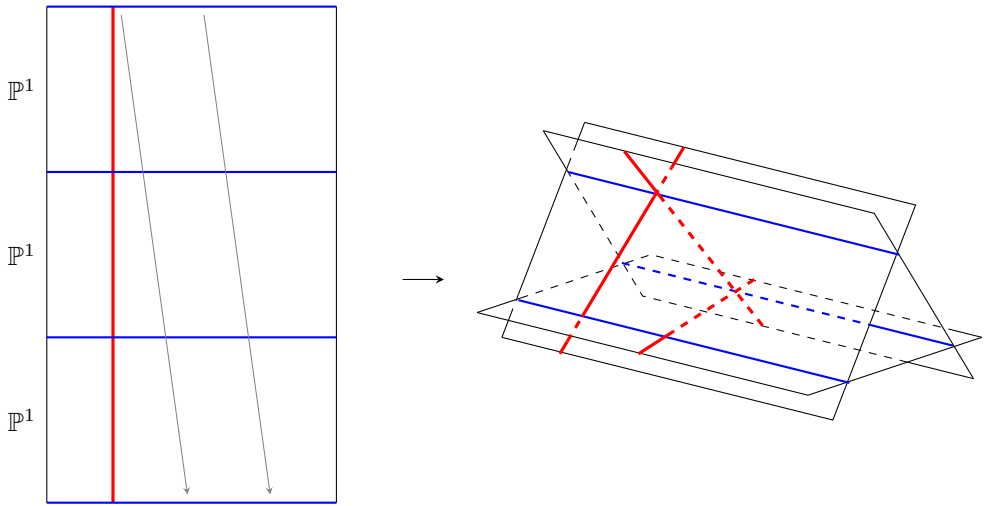
Figure 2: Reducible semistable degeneration 
\end{center}

We expect the following behaviour:
\begin{conjecture}
Let $\pi:X\rightarrow\P^n$ be a very general Lagrangian fibration, and let $t$ be a general point of the hypersurface $\Delta\subset\P^n$ parametrizing singular fibres. Then $X_t$ is a semistable degeneration of abelian varieties.
\end{conjecture}

For K3 surfaces, the conjecture states that a very general elliptic K3 surface has only singular fibres of type $I_k$ in Kodaira's classification. This is certainly true; indeed, a very general elliptic K3 surface will have only $I_1$ singular fibres (rational nodal curves), as can be shown by explicitly finding such a surface and showing that it deforms in a $19$-dimensional family. The challenge is to find a proof of this fact that can be extended to higher dimensional Lagrangian fibrations. Although we do not achieve this, we do find a simple proof of a weaker property of elliptic K3 surfaces that generalizes to give:
\begin{theorem}
\label{main}
Let $\pi:X\rightarrow\P^n$ be a very general Lagrangian fibration, and let $t$ be a general point of the hypersurface $\Delta\subset\P^n$ parametrizing singular fibres. Then $X_t$ is either
\begin{enumerate}
\item an {\'e}tale quotient of an $n$-dimensional abelian variety (such a fibre is singular because it will have multiplicity greater than one),
\item a semistable degeneration of abelian varieties, or
\item a fibration over an $(n-1)$-dimensional abelian variety $A$ by singular elliptic curves of Kodaira type $II$ (a cuspidal rational curve), $III$ (two rational curves meeting at a tacnode), or $IV$ (three rational curves meeting at a point), up to an {\'e}tale cover.
\end{enumerate}
\end{theorem}
Of course, we believe that only the second case is possible. Theorem~\ref{main} has been obtained independently by C.\ Lehn (Theorem~5.7 of~\cite{lehn16}) as a corollary of his generalization of Voisin's results on deformations of Lagrangian submanifolds to deformations of Lagrangian normal crossing subvarieties. Our proof is more direct: a standard argument shows that a very general Lagrangian fibration must have Picard number $\rho(X)=1$, and this places restrictions on the structure of the discriminant hypersurface $\pi^{-1}(\Delta)$ from which the theorem follows.

Note that it is {\em not\/} true that the singular fibres in codimension two (or higher) must be semistable. For example, let $\pi:X\rightarrow\P^n$ be a Beauville-Mukai system~\cite{beauville99, mukai84}, i.e., the compactified relative Jacobian of a complete linear system of curves on a K3 surface. There will be cuspidal curves in codimension two in this linear system, and their compactified Jacobians will not be semistable; see~\cite{sawon15} for details. Moreover, we can make this Lagrangian fibration very general by deforming it in a $20$-dimensional family, inside the $21$-dimensional moduli space of deformations of $X$, while preserving the existence of these non-semistable singular fibres in codimension two.

In Section~2 we recall some facts about singular fibres of elliptic K3 surfaces. In Section~3 we extend these results to higher dimensional Lagrangian fibrations, thereby proving our main result, Theorem~\ref{main}. Section~4 contains a number of illustrative examples.

The author would like to thank Concettina Galati, Andreas Knutsen, and Christian Lehn for helpful conversations, and the Erwin Schr{\"o}dinger Institute for hospitality. The author gratefully acknowledges support from the NSF, grant numbers DMS-1206309 and DMS-1555206.

\section{K3 surfaces}

We begin by recalling Kodaira's classification of singular fibres of elliptic surfaces.

\begin{thm}[see Barth et al.~\cite{bhpv04}]
\label{kodaira}
A non-multiple singular fibre of a minimal elliptic surface must be one of the following types:
\begin{itemize}
\item $I_k$, $k\geq 1$, a cycle of $k$ rational curves with intersection matrix given by the Cartan matrix of the affine Dynkin diagram of type $\tilde{A}_{k-1}$ (with $I_1$ a nodal rational curve),
\item $II$, a cuspidal rational curve,
\item $III$, two rational curves meeting at a tacnode,
\item $IV$, three rational curves meeting at a single point,
\item $I^*_k$, $k\geq 0$, $k+5$ rational curves with intersection matrix given by $\tilde{D}_{k+4}$,
\item $II^*$, nine rational curves with intersection matrix given by $\tilde{E}_8$,
\item $III^*$, eight rational curves with intersection matrix given by $\tilde{E}_7$,
\item $IV^*$, seven rational curves with intersection matrix given by $\tilde{E}_6$.
\end{itemize}
A multiple singular fibre must be of type $mI_0$, a smooth elliptic curve with multiplicity $m\geq 2$, or of type $mI_k$ with $k\geq 1$, i.e., type $I_k$ with multiplicity $m\geq 2$.
\end{thm}

Let $S\rightarrow\P^1$ be an elliptic K3 surface, with fibre class $F\in\H^{1,1}(S)\cap\H^2(S,\Z)$. If we deform $S$ so that $F$ remains algebraic, i.e., of type $(1,1)$, then the resulting K3 surface will still be elliptic. This means that elliptic K3 surfaces are codimension one inside the $20$-dimensional moduli space of all K3 surfaces. Moreover, a very general elliptic K3 surface will have Picard number $\rho=1$; note that it is non-projective and does not admit a section (or even a multi-valued rational section).

\begin{lemma}
Let $S\rightarrow\P^1$ be a very general elliptic K3 surface. Then every singular fibre of $S$ is reduced and irreducible. Thus by Kodaira's classification, every singular fibre is either of type $I_1$ (nodal rational curve) or type $II$ (cuspidal rational curve).
\end{lemma}

\begin{proof} The N{\'e}ron-Severi group of a non-projective elliptic surface is spanned by the fibre class $F$ and the irreducible components of the singular fibres. Suppose there are $k$ singular fibres $S_1,\ldots,S_k$ with $m_1,\ldots,m_k$ irreducible components, respectively. Obviously some linear combination of the components of $S_i$ is linearly equivalent to $F$, so each $S_i$ really contributes $m_i-1$ additional independent classes to the N{\'e}ron-Severi group. Thus the rank of the N{\'e}ron-Severi group is
$$\rho(S)=1+\sum_{i=1}^k (m_i-1).$$
Recall that the Shioda-Tate formula~\cite{shioda72} for an elliptic surface with a section is
$$\rho(S)=2+\mathrm{rank}MW(\pi)+\sum_{i=1}^k(m_i-1),$$
where $MW(\pi)$ is the Mordell-Weil group of sections of $\pi:S\rightarrow\P^1$. Our formula above is essentially a degenerate version of the Shioda-Tate formula for non-projective surfaces. Since a very general elliptic K3 surface has $\rho(S)=1$, we immediately see that $m_i=1$ for all $i$, i.e., every singular fibre of $S$ is irreducible.

Multiple fibres contribute non-trivially to Kodaira's formula for the canonical bundle of an elliptic surface, but K3 surfaces have trivial canonical bundles, and thus they cannot have any multiple fibres. Thus every singular fibre of $S$ is reduced.
\end{proof}

\begin{question}
Although it does not follow from the above arguments, very general elliptic K3 surfaces actually have only nodal rational curves as singular fibres. This can be proved by explicitly constructing such a K3 surface and then showing that there are $19$-parameters describing its deformations. Is there a deformation argument that eliminates cuspidal rational curves? Given an elliptic fibration over a disc with a cuspidal rational curve over $0$, we can deform it to a fibration with two nodal rational curves. However, we need a global argument to show that such a deformation fits in to an elliptic K3 surface.
\end{question}

\section{Higher dimensions}

We want to extend the ideas of the previous section to higher dimensional Lagrangian fibrations, i.e., fibrations $\pi:X\rightarrow\P^n$ where $X$ is an irreducible holomorphic symplectic manifold and the general fibre is an $n$-dimensional abelian variety (see~\cite{sawon03} and the references therein). We start by describing Hwang and Oguiso's~\cite{ho09,ho11} Kodaira-type classification of general singular fibres. By Proposition~3.1(2) of~\cite{ho09} the discriminant locus $\Delta\subset\P^n$ parametrizing singular fibres is a hypersurface. Roughly speaking, the main observation of Hwang and Oguiso is that for a general singular fibre $X_t$, above a general point $t\in\Delta$, there is a residual $(n-1)$-dimensional abelian variety present and the fibre is degenerating in the one remaining dimension. To state this precisely, let $V$ be a component of the reduction $(X_t)_{\mathrm{red}}$ of $X_t$, and let $\hat{V}$ be its normalization. Then the Albanese map $\hat{V}\rightarrow\mathrm{Alb}(\hat{V})$ is a fibre bundle with fibre either $\P^1$ or an elliptic curve. The image of such a fibre in $(X_t)_{\mathrm{red}}$ is called a {\em characteristic leaf\/}; if two (or more) leaves meet they form a {\em characteristic curve\/}, and once we add in multiplicities coming from the multiplicities of the components of $X_t$ we obtain a {\em characteristic $1$-cycle\/}. 

Hwang and Oguiso gave a classification of these characteristic $1$-cycles.
\begin{thm}[Hwang-Oguiso~\cite{ho09,ho11}]
\label{hoclassification}
Let $X_t$ be a general singular fibre with multiplicity $m$. Then all characteristic $1$-cycles $\sum_im_iC_i$ in $X_t$ are isomorphic, and the $1$-cycle $\sum_i\frac{m_i}{m}C_i$ is either
\begin{enumerate}
\item a smooth elliptic curve,
\item one of the singular elliptic fibres of Kodaira's classification, as in Theorem~\ref{kodaira}, or
\item of type $A_{\infty}$, i.e., a $1$-cycle $\sum_{i\in\mathbb{Z}}C_i$ consisting of a chain of infinitely many rational curves $C_i$, with $C_i$ meeting $C_j$ if and only if $j=i\pm 1$. 
\end{enumerate}
\end{thm}

A singular fibre that is a product of an $(n-1)$-dimensional abelian variety and a singular elliptic curve from Kodaira's classification will have characteristic $1$-cycle of the same type as the singular elliptic curve. For example, the product of an $(n-1)$-dimensional abelian variety and an elliptic curve of type $I_2$ will have characteristic $1$-cycle of type $I_2$; such a singular fibre will have two irreducible components. But singular fibres need not be products, and so it is possible for an irreducible singular fibre to have a reducible characteristic $1$-cycle, as illustrated in Figure $3$.

\begin{center}
\vspace*{3mm}
\begin{tikzpicture}[scale=1.0,>=stealth]
  \draw (-2,0) -- (-2,2.5);
  \draw (-6,0) -- (-6,2.5);
  \draw[thick, blue] (-2,0) -- (-6,0);
  \draw[thick, blue] (-2,2.5) -- (-6,2.5);
  \draw[very thick, red] (-5,0) -- (-5,2.5);
  \draw[very thick, red] (-3,0) -- (-3,2.5);
  \draw[very thin, gray, ->] (-4.9,2.4) -- (-3.1,0.1);
  \draw[very thin, gray] (-2.9,2.4) -- (-2.1,1.3);
  \draw[very thin, gray, ->] (-5.9,1.2) -- (-5.1,0.1);
  
  \draw[->] (-1.2,1.2) -- (-0.7,1.2);
  
  \draw (0,0) .. controls (0.2,0.1) and (0.7,0.5) .. (0.8,0.6);
  \draw[dashed] (0.8,0.6) .. controls (1.5,1.3) and (1.5,2.6) .. (0.8,2.7);
  \draw (0.8,2.7) .. controls (0.1,2.5) and (0.1,1.2) .. (0.8,0.6);
  \draw[dashed] (0.8,0.6) .. controls (0.7,0.7) and (1.2,0.3) .. (1.4,0.2);
  
  \draw (4,-1) .. controls (4.2,-0.9) and (4.7,-0.5) .. (4.8,-0.4);
  \draw (4.8,-0.4) .. controls (5.5,0.3) and (5.5,1.6) .. (4.8,1.7);
  \draw (4.8,1.7) .. controls (4.1,1.5) and (4.1,0.2) .. (4.8,-0.4);
  \draw (4.8,-0.4) .. controls (4.7,-0.3) and (5.2,-0.7) .. (5.4,-0.8);
  
  \draw (0,0) -- (4,-1);
  \draw (0.8,2.7) -- (4.8,1.7);
  \draw[thick, blue] (0.8,0.6) -- (4.8,-0.4);
  \draw[dashed] (1.4,0.2) -- (4.6,-0.6);
  \draw (4.6,-0.6) -- (5.4,-0.8);

  \draw[very thick, red] (0.5,-0.125) .. controls (1,0) and (1.2,0) .. (1.8,0.35);
  \draw[dashed, very thick, red] (1.8,0.35) .. controls (2.4,0.7) and (3.3,2.0) .. (2.8,2.2);
  \draw[very thick, red] (2.8,2.2) .. controls (2.0,1.9) and (3.2,0.15) .. (3.8,-0.15);
  \draw[dashed, very thick, red] (3.8,-0.15) .. controls (4.3,-0.45) and (4.3,-0.5) .. (4.7,-0.6);
  
  \draw[very thick, red] (2.5,-0.625) .. controls (3,-0.5) and (3.2,-0.5) .. (3.8,-0.15);
  \draw[dashed, very thick, red] (3.8,-0.15) .. controls (4.4,0.2) and (5.3,1.5) .. (4.8,1.7);
  \draw[very thick, red] (0.8,2.7) .. controls (0.0,2.4) and (1.2,0.65) .. (1.8,0.35);
  \draw[dashed, very thick, red] (1.8,0.35) .. controls (2.3,0.05) and (2.3,0) .. (2.7,-0.1);
  
\end{tikzpicture}\\
\vspace*{5mm}
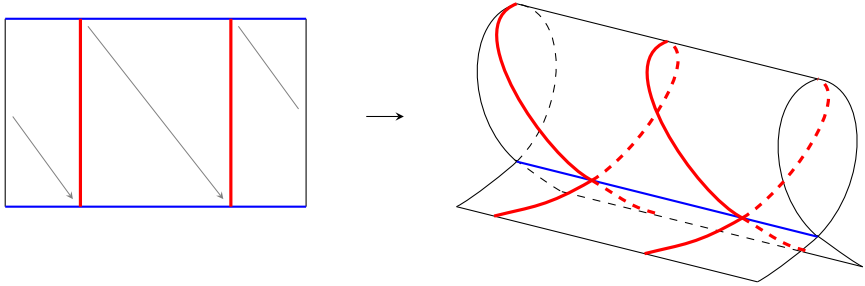
Figure 3: Irreducible singular fibre with characteristic $1$-cycle of type $I_2$ 
\end{center}

\begin{remark}
Unlike elliptic K3 surfaces, the singular fibres of Lagrangian fibrations can have multiplicities. Hwang and Oguiso~\cite{ho11} showed that various values up to and including six are possible, depending on the type of the (characteristic $1$-cycle of the) singular fibre. In particular, it is possible for the (reduced) characteristic $1$-cycle of a {\em singular\/} fibre to be a {\em smooth\/} elliptic curve if the fibre has multiplicity $m>1$. Some examples will be given in Section~\ref{later}.
\end{remark}

\begin{remark}
Of course, a singular fibre cannot have infinitely many components. Thus a characteristic $1$-cycle of type $A_{\infty}$ must `wrap around' the singular fibre. An example of a reduced and irreducible singular fibre with characteristic $1$-cycle of type $A_{\infty}$ will be given in Section~\ref{later2}.
\end{remark}

Let $\pi:X\rightarrow\P^n$ be a Lagrangian fibration with $n\geq 2$. Let $F\in\H^{n,n}(X)\cap\H^{2n}(X,\Z)$ be the fibre class and let $L\in\H^{1,1}(X)\cap\H^2(X,\Z)$ be the pullback of a hyperplane in $\P^n$. Inside the moduli space $\mathcal{M}$ of deformations of $X$ as a complex manifold, there is a hypersurface $\mathcal{H}_F$ parametrizing deformations such that $F$ remains algebraic (see Voisin~\cite{voisin92}) and there is also a hypersurface $\mathcal{H}_L$ parametrizing deformations such that $L$ remains algebraic. In fact, these hypersurfaces are identical and Matsushita~\cite{matsushita05, matsushita16} proved that they parametrize deformations of $X$ that remain Lagrangian fibrations. In a family of deformations of $X$, the Picard number is upper semicontinuous; Oguiso~\cite{oguiso03} proved that the Picard number jumps up on a dense countable subset. More importantly for us, a family parametrized by $\mathcal{M}^{\prime}\subset\mathcal{M}$ of codimension $k$ cannot contain only deformations of $X$ with Picard numbers $>k$. In particular, a family parametrized by a hypersurface in $\mathcal{M}$ must contain deformations with Picard numbers $1$ (or $0$). It follows that a very general Lagrangian fibration will have Picard number $\rho(X)=1$; it will also be non-projective and will not admit a (multi-valued rational) section.

Our main result, Theorem~\ref{main}, follows directly from the following proposition: the three cases for the characteristic $1$-cycle correspond directly to the three cases for the structure of the singular fibre in Theorem~\ref{main}.
\begin{proposition}
\label{prop}
Let $\pi:X\rightarrow\P^n$ be a very general Lagrangian fibration, and let $t$ be a general point of the discriminant locus $\Delta\subset\P^n$. Then the (reduced) characteristic $1$-cycle of the singular fibre $X_t$ is either
\begin{enumerate}
\item a smooth elliptic curve,
\item a singular elliptic curve of Kodaira type $I_k$ with $k\geq 1$, or of type $A_{\infty}$, or
\item a singular elliptic curve of Kodaira type $II$, $III$, or $IV$.
\end{enumerate}
\end{proposition}

\begin{remark}
In the first case, $X_t$ must have multiplicity $m=2,3,4,$ or $6$ (not $1$ or it would be a smooth fibre). In the second case, we must have $m=1$ for type $I_k$ with $k$ odd, and $m=1$ or $2$ for type $I_k$ with $k$ even and type $A_{\infty}$. In the third case, we must have $m=1$ or $5$ for type $II$, $m=1$ or $3$ for type $III$, and $m=1$ or $2$ for type $IV$ (see Hwang and Oguiso~\cite{ho11}).
\end{remark}

Before presenting the proof of Proposition~\ref{prop}, we state a generalization of the Shioda-Tate formula to higher dimensional fibrations by abelian varieties.
\begin{thm}[Oguiso~\cite{oguiso09}, Theorem~1.1]
Let $\varphi:X\rightarrow Y$ be a proper surjective morphism with rational section $O$, whose generic fibre $A:=X_{\eta}$ is an abelian variety defined over the field $K=\C(Y)$ with origin $O$. Assume further that $X$ and $Y$ have  only $\Q$-factorial rational singularities, $\varphi$ is equi-dimensional in codimension one, and $h^1(X,\O_X)=h^1(Y,\O_Y)$. Write $\Delta=\cup_{i=1}^k\Delta_i$ for the decomposition into irreducible components of the discriminant divisor $\Delta\subset Y$, and assume that $\varphi^{-1}(\Delta_i)\subset X$ consists of $m_i$ irreducible components. Then the Mordell-Weil group $MW(\varphi)$, i.e., the group $A(K)$ of $K$-rational points of $A$ (equivalently, the group of rational sections of $\varphi$), is a finitely generated abelian group of rank
$$\mathrm{rank}MW(\varphi)=\rho(X)-\rho(Y)-\mathrm{rank}NS(A_K)-\sum_{i=1}^k(m_i-1).$$
\end{thm}

\begin{proof}{\bf of Proposition~\ref{prop}} We apply Oguiso's theorem to a Lagrangian fibration $\pi:X\rightarrow\P^n$. In this case, $X$ and $Y=\P^n$ are both smooth, Lagrangian fibrations are equi-dimensional, and $h^1(X,\O_X)=h^1(\P^n,\O_{\P^n})=0$. The theorem assumes the existence of a (rational) section. In~\cite{sawon09} the author proved that there is a hypersurface in $\mathcal{H}_F=\mathcal{H}_L$ parametrizing Lagrangian fibrations that admit sections; in other words, Lagrangian fibrations that admit sections are codimension two in $\mathcal{M}$. A very general Lagrangian fibration that admits a section will therefore have Picard number $\rho(X)=2$. In addition, the existence of a section implies that $X$ is projective and hence the N{\'e}ron-Severi group $NS(A_K)$ must have rank at least one. The Shioda-Tate formula,
$$\mathrm{rank}MW(\pi)=2-1-\mathrm{rank}NS(A_K)-\sum_{i=1}^k(m_i-1),$$
then forces $\mathrm{rank}MW(\pi)=0$, $\mathrm{rank}NS(A_K)=1$, and $m_i=1$ for all $i$. If we consider instead a Lagrangian fibration $\pi:X\rightarrow\P^n$ that does not admit any (multi-valued rational) section, then there is an obvious analogue of the formula in which $\mathrm{rank}MW(\pi)$ and $\mathrm{rank}NS(A_K)$ both vanish and
$$\rho(X)=\rho(\P^n)+\sum_{i=1}^k(m_i-1).$$
For a very general Lagrangian fibration $\rho(X)=1$ and again we find that $m_i=1$ for all $i$. Summarizing, we have proved that if the discriminant divisor $\Delta\subset\P^n$ of a very general Lagrangian fibration $\pi:X\rightarrow\P^n$ decomposes into irreducible components as $\Delta=\cup_{i=1}^k\Delta_i$, then $\pi^{-1}(\Delta_i)$ is irreducible for all $i$.

Next, let $t$ be a general point of $\Delta_i\subset\Delta$. The statement above does {\em not\/} imply that the fibre $X_t$ is irreducible. It is possible that the fibre $X_t$ has several irreducible components $Y_1,\ldots,Y_l$ that are all contained in the single irreducible divisor $\pi^{-1}(\Delta_i)$, because of `monodromy' permuting the components as we move around a loop in $\Delta_i$ starting and ending at $t$, avoiding the codimension one subset $\Delta_{i0}\subset\Delta_i$ parametrizing non-general singular fibres. However, we see that this monodromy representation
$$\pi_1(\Delta_i\backslash\Delta_{i0},t)\longrightarrow \mathrm{Sym}_l$$
must act transitively on the set of components of $X_t$. This implies that the components $Y_j$ must all have the same multiplicity. Referring to Theorem~\ref{hoclassification}, we can conclude that the (reduced) characteristic $1$-cycle of $X_t$ must belong to one of the three cases of the proposition. The other possibilities in Hwang and Oguiso's classification are singular elliptic curves of Kodaira type $I^*_m$ with $m\geq 0$, $II^*$, $III^*$, and $IV^*$, but these all have components with different multiplicities, which is not allowed.
\end{proof}

\section{Examples}

In this section we describe the singular fibres in codimension one for several examples of Lagrangian fibrations.

\subsection{Multiple fibres with smooth reduction}
\label{later}

In Example~6.2 of~\cite{ho11}, Hwang and Oguiso described several examples where the singular fibre has smooth reduction, which is equivalent to the (reduced) characteristic $1$-cycle being a smooth elliptic curve. Their examples are local Lagrangian fibrations, but the same constructions easily extend to give global compact Lagrangian fibrations. For instance, let us extend Example~6.2(v) which has singular fibres of multiplicity $2$ and characteristic $1$-cycle an arbitrary smooth elliptic curve $E_1$ with equation $y^2=x^3+ax+b$. Let $E_2$, $E_3$, and $E_4$ be arbitrary smooth elliptic curves with coordinates $t$, $z$, and $s$, respectively, and let $p_2$ be a $2$-torsion point on $E_3$. Then define $X$ to be the quotient of $E_1\times E_2\times E_3\times E_4$ by the fixed-point-free involution
$$g^*:((x,y),t,z,s)\longmapsto ((x,-y),-t,z+p_2,s).$$
The symplectic form $\sigma:=\frac{dx\wedge dt}{y}+dz\wedge ds$ on $E_1\times E_2\times E_3\times E_4$ is preserved by $g^*$ and therefore descends to $X$. The projection 
\begin{eqnarray*}
X & \longrightarrow & E_2/\pm 1\times E_4\cong{\mathbb{P}}^1\times E_4 \\
{[((x,y),t,z,s)]} & \longmapsto & (\pm t,s)
\end{eqnarray*}
makes $X$ into a Lagrangian fibration. The singular fibres sit above $q_2\times E_4$, where $q_2$ is a $2$-torsion point in $E_2$ (fixed by $\pm 1$), and they look like the hyperelliptic surface
$$E_1\times E_3/((x,y),z)\sim ((x,-y),z+p_2)$$
with multiplicity $2$.

This example is an isotrivial Lagrangian fibration: all the smooth fibres are isomorphic to $E_1\times E_3$. We can construct a non-isotrivial example by replacing $E_1\times E_2$ by a certain elliptic K3 surface $S$. Specifically, choose for $S$ an elliptic K3 surface admitting a symplectic involution $\tau$ which acts as $\pm 1$ on each fibre and as $\pm 1$ on the base $\P^1=\C\cup\{\infty\}$. The only fibres of $S\rightarrow\P^1$ that are fixed by $\tau$ are those above $0$ and $\infty$; we assume these are smooth.
We can now modify the example above by defining $X$ to be the quotient of $S\times E_3\times E_4$ by the fixed-point-free involution
$$g^*:(q,z,s)\longmapsto (\tau(q),z+p_2,s).$$
The symplectic form $\sigma:=\sigma_S+dz\wedge ds$ descends to $X$ and the projection to $\P^1/\pm 1\times E_4$ makes $X$ into a Lagrangian fibration. The singular fibres sit above $\{0\}\times E_4$ and $\{\infty\}\times E_4$ and look like the hyperelliptic surfaces
$$S_0\times E_3/(w_0,z)\sim(-w_0,z+p_2)\qquad\mbox{and}\qquad S_{\infty}\times E_3/(w_{\infty},z)\sim(-w_{\infty},z+p_2)$$
with multiplicity $2$, where $S_0$ and $S_{\infty}$ are the elliptic fibres of $S\rightarrow\P^1$ above $0$ and $\infty$, respectively, with coordinates $w_0$ and $w_{\infty}$.

Note that a theorem of Hwang~\cite{hwang08} asserts that the base of a Lagrangian fibration of an irreducible holomorphic symplectic manifold must be isomorphic to projective space, if it is smooth. However, in the examples above $X$ is not an {\em irreducible\/} holomorphic symplectic manifold; indeed, it admits an {\'e}tale cover that is a product, $E_1\times E_2\times E_3\times E_4$ and $S\times E_3\times E_4$, respectively. Consequently, bases that are not isomorphic to $\P^2$ can arise.

\subsection{Beauville-Mukai systems}
\label{later2}

Recall the construction of the Beauville-Mukai integrable system~\cite{beauville99, mukai84}. Let $S$ be a K3 surface containing a smooth genus $n$ curve $C$. Then $C$ moves in an $n$-dimensional linear system, $|C|\cong\P^n$; denote by $\mathcal{C}\rightarrow\P^n$ the corresponding family of curves. If $S$ is very general in the sense that its N{\'e}ron-Severi group $NS(S)$ is generated over $\mathbb{Z}$ by $[C]$ then every curve in the family $\mathcal{C}$ is reduced and irreducible. This means that we can apply the Altman and Kleiman construction~\cite{ak80} of the compactified relative Jacobian to obtain $X:=\overline{\mathrm{Jac}}^d(\mathcal{C}/\P^n)$. We can also regard $X$ as a Mukai moduli space of stable sheaves on $S$ with Mukai vector
$$v=(0,[C],d+1-n).$$
Here we think of an element of $X$ as a torsion sheaf $\iota_*L$, where $\iota:C\hookrightarrow S$ is the embedding of a curve into the K3 surface and $L$ is a degree $d$ line bundle (or more generally, a rank one torsion-free sheaf) on $C$; then $\iota_*L$ is stable in the sense of Simpson~\cite{simpson94}. This latter point of view shows that $X$ admits a holomorphic symplectic structure, and $X\rightarrow\P^n$ is therefore a Lagrangian fibration. Moreover, the definition of $X$ as a Mukai moduli space makes sense even when $NS(S)\not\cong\mathbb{Z}.[C]$, and provided the Mukai vector $v$ is isotropic and a general polarization of $S$ is chosen, we will once again obtain a (smooth, compact) Lagrangian fibration $X\rightarrow\P^n$.

For a very general (polarized) K3 surface, a general codimension one singular curve in the family $\mathcal{C}$ will have a single node. This can be proved by studying the Beauville-Mukai system and using properties of Lagrangian fibrations (see Lemma~2.4 of~\cite{sawon15}); it is also a special case of Lemma~3.1 of Chen~\cite{chen99}. Let $C$ be such a curve, with normalization $\tilde{C}$ of genus $n-1$, and $C$ obtained by identifying two points $p$ and $q\in\tilde{C}$. The normalization of the compactified Jacobian $\overline{\mathrm{Jac}}^dC$ of $C$ is a $\P^1$-bundle over the Jacobian $\mathrm{Jac}^d\tilde{C}$ of $\tilde{C}$, and the compactified Jacobian $\overline{\mathrm{Jac}}^dC$ itself is obtained by identifying the $0$- and $\infty$-sections of this $\P^1$-bundle via a translation
$$0\mbox{-section}\cong\mathrm{Jac}^d\tilde{C}\stackrel{\otimes\O(p-q)}{\longrightarrow}\mathrm{Jac}^d\tilde{C}\cong \infty\mbox{-section}.$$
Since the characteristic leaves are the images of the $\P^1$-fibres, the type of the characteristic $1$-cycle therefore depends on whether or not $\O(p-q)$ is a torsion line bundle in $\mathrm{Jac}^0\tilde{C}$.

\begin{lemma}
The characteristic $1$-cycle of a general singular fibre of a Beauville-Mukai system constructed from a very general K3 surface is of type $A_{\infty}$.
\end{lemma}

\begin{remark}
A local example of a Lagrangian fibration with general fibres of type $A_{\infty}$ was described by Matsushita (see Proposition~4.13 in~\cite{ho09}).
\end{remark}

\begin{proof}
Consider first the genus $n=2$ case, studied in detail in~\cite{sawon08ii}. We can write the discriminant locus as a disjoint union
$$\Delta=(\Delta\backslash\Delta_0)\cup \Delta_0=(\Delta\backslash\Delta_0)\cup\Delta_c\cup\Delta_{nn},$$
where $\Delta\backslash\Delta_0$ parametrizes curves with one node, $\Delta_c$ parametrizes curves with one cusp, and $\Delta_{nn}$ parametrizes curves with two nodes. Note that $\Delta_c$ is the set of 72 cusps of $\Delta$ and $\Delta_{nn}$ is the set of 324 nodes of $\Delta$. Consider the normalization $\tilde{\mathcal{C}}\rightarrow (\Delta\backslash\Delta_0)$ of the family of nodal curves $\mathcal{C}\rightarrow (\Delta\backslash\Delta_0)$. Because the cuspidal curves over $\Delta_c$ also have normalizations of genus one, this family extends to a family $\tilde{\mathcal{C}}\rightarrow (\Delta\backslash\Delta_0)\cup\Delta_c$ of smooth genus one curves. The family of line bundles $\O(p-q)$ then defines a section of the relative Jacobian
$$\mathrm{Jac}^0(\tilde{\mathcal{C}}/(\Delta\backslash\Delta_0))$$
over $\Delta\backslash\Delta_0$. As we approach a cuspidal curve, $p\rightarrow q$, so this section extends to a section of the relative Jacobian
$$\mathrm{Jac}^0(\tilde{\mathcal{C}}/(\Delta\backslash\Delta_0)\cup\Delta_c).$$
Moreover, this section intersects the zero section, i.e., takes the value $\O$, precisely above $\Delta_c$, because $\O(p-q)=\O(p-p)=\O$ only for a cuspidal curve. In particular, the section is non-zero and not constant. It follows that for a general curve in $\Delta\backslash\Delta_0$, with a single node, $\O(p-q)$ will not be torsion in $\mathrm{Jac}^0\tilde{C}$, and therefore the compactified Jacobian $\overline{\mathrm{Jac}}^dC$ will be of type $A_{\infty}$, i.e., the characteristic $1$-cycle will consist of an infinite chain of rational curves.

The same argument works for higher genus curves. For a very general K3 surface $S$, Galati and Knutsen proved that in the $n$-dimensional linear system $|C|\cong\P^n$ there will exist a curve with a cusp and $n-2$ nodes (Theorem~1.1 in~\cite{gk14}). Moreover, they showed that the $n-2$ nodes may be smoothed independently, producing an $(n-2)$-dimensional family of curves with single cusps. In other words, $\Delta\backslash\Delta_0$ parametrizes curves with one node, and there again exists a non-empty codimension one subset $\Delta_c\subset\Delta$ (i.e., codimension two in $|C|$) parametrizing cuspidal curves. The proof for $n>2$ then proceeds as above.
\end{proof}

\begin{remark}
The section of $\mathrm{Jac}^0(\tilde{\mathcal{C}}/(\Delta\backslash\Delta_0)\cup\Delta_c)$ in the above proof will also intersect the (multi-)section of $k$-torsion points, for any $k\geq 2$. At these points of intersection $\O(p-q)$ will be $k$-torsion and therefore the compactified Jacobian $\overline{\mathrm{Jac}}^dC$ will be of type $I_k$, i.e., the characteristic $1$-cycle will consist of $k$ rational curves forming a cycle. Thus the Beauville-Mukai system contains singular fibres corresponding to characteristic $1$-cycles of all types $I_k$ for $k\geq 2$ and $A_{\infty}$. There will be no singular fibres corresponding to characteristic $1$-cycles of type $I_1$, as these would require $p=q$, which as we saw only occurs for cuspidal curves over $\Delta_c$.
\end{remark}

\subsection{Hilbert schemes of elliptic K3 surfaces}

Next we consider an elliptic K3 surface $S\rightarrow\P^1$. Assume $S\rightarrow\P^1$ admits a section, but is otherwise very general in the sense that it contains exactly 24 nodal rational curves as singular fibres, above the points $p_1,\ldots,p_{24}\in\P^1$. The Hilbert scheme $\mathrm{Hilb}^nS$ of $n$ points on $S$ is an irreducible holomorphic symplectic manifold~\cite{beauville83}, and the elliptic fibration on $S$ induces a Lagrangian fibration
$$\mathrm{Hilb}^nS\longrightarrow\mathrm{Sym}^nS\longrightarrow\mathrm{Sym}^n\P^1=\P^n,$$
where the first map is the Hilbert-Chow morphism.

The fibre over a general point of $\mathrm{Sym}^n\P^1$, given by an $n$-tuple $\{x_1,\ldots,x_n\}$ of {\em distinct\/} points on $\P^1$, is isomorphic to the product $S_{x_1}\times\cdots\times S_{x_n}$ of the corresponding elliptic fibres of $S\rightarrow\P^1$. The singular fibres occur over the hyperplanes
$$\Delta_i:=\{\{x_1,\ldots,x_n\}\in\mathrm{Sym}^n\P^1\; |\;x_j=p_i\mbox{ for some }j\}$$
for $i=1,\ldots,24$, and over the `big diagonal'
$$\Delta_0:=\{\{x_1,\ldots,x_n\}\in\mathrm{Sym}^n\P^1\; |\;x_j=x_k\mbox{ for some }j\mbox{ and }k\}.$$
Consider the former; moreover, without loss of generality, let $\{x_1,\ldots,x_n\}$ be a general point of $\Delta_1$ with $x_1=p_1$, $x_j\neq p_i$ for all $j\geq 2$ and $i$, and $x_j\neq x_k$ for all $j$ and $k$. Then the singular fibre over $\{x_1,\ldots,x_n\}$ will be isomorphic to
$$S_{x_1}\times S_{x_2}\times\cdots\times S_{x_n}\cong S_{p_1}\times S_{x_2}\times\cdots\times S_{x_n} $$
where $S_{p_1}$ is a nodal rational curve and $S_{x_2},\ldots,S_{x_n}$ are smooth elliptic curves. This is semistable and the characteristic $1$-cycle is clearly of type $I_1$.

Next consider a general point $\{x_1,\ldots,x_n\}$ of $\Delta_0$; without loss of generality, assume $x_1=x_2$, $x_j\neq x_k$ otherwise, and $x_j\neq p_i$ for all $j$ and $i$. For simplicity, we first consider the $n=2$ case. Write $E$ for the elliptic curve $S_{x_1}$. The Hilbert scheme $\mathrm{Hilb}^2S$ is obtained from $\mathrm{Sym}^2S$ by blowing up the diagonal; thus each point of the diagonal in $\mathrm{Sym}^2S$ is replaced by a $\P^1$. The singular fibre over $\{x_1=x_2\}\in\Delta_0$ is therefore isomorphic to the union of $\mathrm{Sym}^2E$ and a $\P^1$-bundle over $D$, where $D\cong E$ is the diagonal in $\mathrm{Sym}^2E$. Now $\mathrm{Sym}^2E$ is also a $\P^1$-bundle over $E$, because we have the Abel-Jacobi map
$$\mathrm{Sym}^2E\longrightarrow\mathrm{Pic}^2E\cong E$$
taking a degree two divisor to its corresponding line bundle. The diagonal $D$ is a $4$-valued section of this $\P^1$-bundle; for instance, it intersects the fibre above $\O(2y_0)\in\mathrm{Pic}^2E$ in the four points $2y_1\in D\subset\mathrm{Sym}^2E$ where $y_1-y_0$ is $2$-torsion in $E$. Instead, one could observe that the composition of the maps
$$E\cong D\hookrightarrow \mathrm{Sym}^2E\longrightarrow\mathrm{Pic}^2E$$
takes $y$ to $\mathcal{O}(2y)$, i.e., is multiplication by $2$. Putting everything together, we see that the singular fibre over $\{x_1=x_2\}$ will have characteristic $1$-cycle of type $I^*_0$, i.e., a central $\P^1$ with four other $\P^1$s attached in a $\tilde{D}_4$ configuration. 

For $n>2$ we simply take the product of the above with the smooth $(n-2)$-dimensional abelian variety $S_{x_3}\times\cdots\times S_{x_n}$, so the characteristic $1$-cycle is still of type $I^*_0$. In particular, the singular fibres over $\Delta_0$ are never semistable. However, we will see in the next subsection that they become semistable after a small deformation of the Lagrangian fibration.

Thinking of $\P^n$ as parametrizing degree $n$ polynomials, $\Delta_0$ is precisely the hypersurface parametrizing polynomials with a repeated root. A resultant argument then show that $\Delta_0$ has degree $2(n-1)$, but it seems that one should attach a multiplicity to $\Delta_0$ because it parametrizes non-semistable singular fibres. We will say more about this in the next subsection.

\subsection{Deforming to a very general Lagrangian fibration}

We will now show that the Hilbert scheme of an elliptic K3 surface can be deformed, as a Lagrangian fibration, to a Beauville-Mukai integrable system. Under this deformation, the non-semistable singular fibres of the Hilbert scheme become semistable fibres of the Beauville-Mukai system. We first describe this in dimension $2n=4$, and then in higher dimensions $2n\geq 6$.

Let $S$ be an elliptic K3 surface admitting a section $D$, and assume that the N{\'e}ron-Severi group is generated by $D$ and the fibre $F$. In what follows we will often abbreviate $\H^0(S,\O(D))$ as $\H^0(S,D)$, etc.

\begin{lemma}
There is a natural isomorphism
$$\H^0(S,2F+D)\cong\H^0(S,2F)\cong\C^3.$$
Thus $D$ is the base locus of the linear system $|2F+D|$, the movable part consists of pairs of fibres, and
$$|2F|=\mathrm{Sym}^2\P^1\cong\P^2$$
because each fibre moves in a pencil.
\end{lemma}

\begin{proof}
Firstly
$$(2F+D).F=1,\qquad (2F+D).D=0,\qquad\mbox{and}\qquad (2F+D)^2=2,$$
so $2F+D$ is nef and big. By the Kodaira vanishing theorem and the fact that $S$ has trivial canonical bundle, we deduce that $\H^1(S,2F+D)=0$. The short exact sequence
$$0\longrightarrow \O(2F)\longrightarrow \O(2F+D)\longrightarrow \O(2F+D)|_D\cong\O_D\longrightarrow 0$$
then yields the long exact sequence
$$0\rightarrow\H^0(S,2F)\rightarrow\H^0(S,2F+D)\rightarrow\H^0(D,\O_D)\cong\C\rightarrow\H^1(S,2F)\rightarrow 0.$$
The Leray spectral sequence for $\pi:S\rightarrow\P^1$ yields the exact sequence
$$\H^1(\P^1,\O(2))=0\longrightarrow \H^1(S,2F)=\H^1(S,\pi^*\O(2))\longrightarrow \H^0(\P^1,R^1\pi_*\pi^*\O(2))\longrightarrow 0,$$
which implies
$$\H^1(S,2F)\cong \H^0(\P^1,R^1\pi_*\pi^*\O(2))= \H^0(\P^1,\O(2)\otimes R^1\pi_*\O_S)=\H^0(\P^1,\O)\cong\C,$$
because $R^1\pi_*\O_S\cong\O(-2)$ for an elliptic K3 surface. The long exact sequence above now looks like
$$0\longrightarrow\H^0(S,2F)\longrightarrow\H^0(S,2F+D)\longrightarrow\C\stackrel{\cong}{\longrightarrow}\C\longrightarrow 0.$$
We conclude that the first map is also an isomorphism, i.e.,
$$\H^0(S,2F+D)\cong\H^0(S,2F).$$
Of course, the Leray spectral sequence also yields
$$\H^0(S,2F)=\H^0(S,\pi^*\O(2))\cong\H^0(\P^1,\O(2))\cong\C^3.$$
\end{proof}

The first order deformations of $S$ are parametrized by
$$\H^1(S,T)\cong \H^1(S,\Omega^1)\cong \H^{1,1}(S),$$
where the first isomorphism is induced by interior product with the holomorphic symplectic form $\sigma$, which gives an isomorphism $T\cong\Omega^1$. In terms of periods, deforming in the direction $\alpha\in\H^{1,1}(S)$ changes $\sigma$ into
$$\sigma_t=\sigma+t\alpha$$
to first order in $t$. We can globalize this by adding a quadratic term and defining
$$\sigma_t:=\sigma+t\alpha-t^2\left(\frac{\alpha^2}{2\sigma\bar{\sigma}}\right)\bar{\sigma}.$$
Then $\sigma_t^2=0$ and
$$\sigma_t\bar{\sigma}_t=\left(1+\frac{|t|^2\alpha^2}{2\sigma\bar{\sigma}}\right)^2\sigma\bar{\sigma}$$
is positive for all $t$ sufficiently close to $0$, thus giving a family of K3 surfaces $\mathcal{S}\rightarrow\Delta$ defined over a disc $\Delta$ of some radius.

Consider the deformation in the direction $\alpha=D$, and assume that $\sigma$ is normalized so that $\sigma\bar{\sigma}=1$. Then $\sigma_t=\sigma+tD+t^2\bar{\sigma}$, $\sigma_t\bar{\sigma}_t=(1-|t|^2)\sigma\bar{\sigma}$, and thus our family $\mathcal{S}\rightarrow\Delta$ is defined over a disc of radius $1$. Now
$$\sigma_t.(2F+D)=\sigma.(2F+D)+tD.(2F+D)+t^2\bar{\sigma}.(2F+D)=0,$$
so $2F+D$ remains of type $(1,1)$ on $S_t$ for all $t$. Moreover, for very general $t$ the N{\'e}ron-Severi group of $S_t$ is generated by $2F+D$, and thus $2F+D$ will be ample on $S_t$. The linear system $|2F+D|$ gives a relative map
$$\mathcal{S}/\Delta\longrightarrow\P^2_{\Delta}$$
over $\Delta$. For very general $t$ this map is a branched double cover of $\P^2$, while for $t=0$ the image of $S=S_0\rightarrow\P^2$ is a conic. The drop in dimension of the image at $t=0$ occurs because the linear system $|2F+D|$ has base locus $D$ on $S=S_0$, whereas for $t\neq 0$ the curve $D$ disappears, i.e., is no longer of type $(1,1)$.

Next we define a relative moduli space $\mathcal{M}/{\Delta}$ of sheaves on $\mathcal{S}/\Delta$ by defining $M_t$ to be the Mukai moduli space of stable sheaves on $S_t$ with Mukai vector
$$v=(0,[2F+D],-1).$$
A general element of $M_t$ will look like $\iota_*L$, where $\iota:C\hookrightarrow S_t$ is the inclusion of a curve $C$ in the linear system $|2F+D|$ and $L$ is a degree zero line bundle on $C$. (In general, the Mukai vector of $\iota_*L$ is $(0,[C],d+1-g)$ for a degree $d$ line bundle $L$ on a curve $C$ of genus $g$.)

\begin{proposition}
The relative moduli space $\mathcal{M}/\Delta$ is a family of Lagrangian fibrations over $\P^2_{\Delta}$.  Specifically
\begin{enumerate}
\item for very general $t$, $M_t$ is a Beauville-Mukai integrable system,
\item for $t=0$, $M_0$ is birational to $\mathrm{Hilb}^2S_0$ and the Lagrangian fibration is induced by the original elliptic fibration on $S_0=S$.
\end{enumerate}
\end{proposition}

\begin{remark}
The relative map $\mathcal{S}/\Delta\rightarrow\P^2_{\Delta}$ comes from maps to $|2F+D|^{\vee}$, whereas the relative Lagrangian fibration $\mathcal{M}/\Delta\rightarrow\P^2_{\Delta}$ comes from mapping a sheaf to its support in $|2F+D|$. Thus the $\P^2_{\Delta}$ in the proposition is really dual to the earlier $\P^2_{\Delta}$.
\end{remark}

\begin{proof}
For very general $t$, $M_t$ is a Beauville-Mukai system by definition; so there is nothing to prove in part 1. When $t=0$, the linear system $|2F+D|$ has $D$ as a base locus. Thus every curve $C$ in the linear system $|2F+D|$ looks like $F_1+F_2+D$, where $F_1$ and $F_2$ are a pair of fibres of the elliptic fibration $S\rightarrow\P^1$. Assume that $F_1$ and $F_2$ are smooth and distinct. A degree zero line bundle $L$ on $F_1+F_2+D$ is given by line bundles $L_1$, $L_2$, and $L_3$ on $F_1$, $F_2$, and $D$, respectively, plus isomorphisms $(L_1)_p\cong (L_3)_p$ at $p=F_1\cap D$ and $(L_2)_q\cong (L_3)_q$ at $q=F_2\cap D$. The degrees $d_1$, $d_2$, and $d_3$ of $L_1$, $L_2$, and $L_3$ must sum to zero. In fact, stability of $\iota_*L$ implies that $d_1=d_2=d_3=0$. The only degree zero line bundle on $D\cong\P^1$ is the trivial line bundle, whereas degree zero line bundles on $F_1$ and $F_2$ are parametrized by $\mathrm{Pic}^0F_1\cong F_1$ and $\mathrm{Pic}^0F_2\cong F_2$, respectively. Therefore the fibre of $M_0\rightarrow |2F+D|$ over the point $F_1+F_2+D$ is isomorphic to $F_1\times F_2$.

Similarly, the elliptic fibration $S_0\rightarrow\P^1$ induces a Lagrangian fibration
$$\mathrm{Hilb}^2S_0\longrightarrow\mathrm{Sym}^2\P^1\cong\P^2$$
whose fibre over a general point $F_1+F_2\in |2F|=\mathrm{Sym}^2\P^1\cong\P^2$ is also isomorphic to $F_1\times F_2$. Since these Lagrangian fibrations both admit global sections, they must be birational.
\end{proof}

\begin{remark}
The argument easily extends to some singular fibres. For example, suppose that $F_1$ is a nodal rational curve. Then we should allow $L_1$ to be a rank one torsion-free sheaf on $F_1$. In other words, we should replace $\mathrm{Pic}^0F_1$ by the compactified Jacobian $\overline{\mathrm{Pic}}^0F_1$. However, $\overline{\mathrm{Pic}}^0F_1$ is isomorphic to $F_1$, as required. The case $F_1=F_2$, even for $F_1$ smooth, appears to be more complicated.
\end{remark}

\begin{remark}
For each fibre, we identified $\mathrm{Pic}^0F$ with $F$. This identification is fixed by taking $\mathcal{O}_F$ to the basepoint in $F$ given by the intersection $F\cap D$ with the section $D$. Alternatively, we could work with the moduli space $M(0,[C],1)$ whose general element looks like $\iota_*L$ for a degree {\em two\/} line bundle $L$. Then stability will force $L$ to come from degree one line bundles on $F_1$ and $F_2$ and a degree zero line bundle on $D$. The fibre over $F_1+F_2+D$ would then be
$$\mathrm{Pic}^1F_1\times\mathrm{Pic}^1F_2\cong F_1\times F_2,$$
and this isomorphism is canonical: we don't even need a basepoint in each fibre.
\end{remark}

\begin{corollary}
The Hilbert scheme $\mathrm{Hilb}^2S_0$ can be deformed to a Beauville-Mukai integrable system, as a Lagrangian fibration over $\P^2$.
\end{corollary}

\begin{proof}
This is achieved by a birational modification of the relative moduli space $\mathcal{M}/\Delta$. We showed that the fibre $M_0$ above $t=0$ is birational to $\mathrm{Hilb}^2S_0$. Birational holomorphic symplectic manifolds correspond to non-separated points in the moduli space (see Huybrechts~\cite{huybrechts97}), so we can replace $M_0$ by $\mathrm{Hilb}^2S_0$ in a family.
\end{proof}

Assume that $S$ is very general, so that it contains exactly $24$ nodal rational curves as singular fibres. The Lagrangian fibration on $\mathrm{Hilb}^2S_0$ has a discriminant locus consisting of $24$ lines and a conic, with the lines tangent to a conic; the singular fibres above the lines are semistable, whereas those above the conic are not. A Beauville-Mukai integrable system will have a discriminant locus of degree $30$ (see~\cite{sawon08i}), and general singular fibres will be semistable. Thus the corollary demonstrates how non-semistable fibres can deform to semistable fibres under a small deformation of the Lagrangian fibration. In addition, it suggests that the conic should be counted with multiplicity three, so that the degree of the discriminant locus is preserved under deformations.

We now extend this example to higher dimensions.

\begin{lemma}
For $n\geq 3$ there is a natural isomorphism
$$\H^0(nF+D)\cong\H^0(nF)\cong\C^{n+1}.$$
Thus $D$ is the base locus of the linear system $|nF+D|$, the movable part consists of sets of $n$ fibres, and 
$$|nF|=\mathrm{Sym}^n\P^1\cong\P^n.$$
\end{lemma}

\begin{proof}
Firstly
$$(nF+D).F=1,\qquad (nF+D).D=n-2>0,\qquad\mbox{and}\qquad (nF+D)^2=2n-2>0,$$
so $nF+D$ is ample. By the Kodaira vanishing theorem $\H^1(S,nF+D)=0$. The short exact sequence
$$0\longrightarrow \O(nF)\longrightarrow \O(nF+D)\longrightarrow \O(nF+D)|_D\cong\O_D(n-2)\longrightarrow 0$$
then yields the long exact sequence
$$0\rightarrow\H^0(S,nF)\rightarrow\H^0(S,nF+D)\rightarrow\H^0(D,\O_D(n-2))\cong\C^{n-1}\rightarrow\H^1(S,nF)\rightarrow 0.$$
The Leray spectral sequence for $\pi:S\rightarrow\P^1$ yields the exact sequence
$$\H^1(\P^1,\O(n))=0\longrightarrow \H^1(S,nF)=\H^1(S,\pi^*\O(n))\longrightarrow \H^0(\P^1,R^1\pi_*\pi^*\O(n))\longrightarrow 0,$$
which implies
$$\H^1(S,nF)\cong \H^0(\P^1,R^1\pi_*\pi^*\O(n))= \H^0(\P^1,\O(n)\otimes R^1\pi_*\O_S)=\H^0(\P^1,\O(n-2))\cong\C^{n-1}.$$
The long exact sequence above now looks like
$$0\longrightarrow\H^0(S,nF)\longrightarrow\H^0(S,nF+D)\longrightarrow\C^{n-1}\stackrel{\cong}{\longrightarrow}\C^{n-1}\longrightarrow 0.$$
We conclude that the first map is also an isomorphism, i.e.,
$$\H^0(S,nF+D)\cong\H^0(S,nF).$$
Thus $nF+D$ is ample but not very ample: it has base locus $D$. Finally, the Leray spectral sequence also yields
$$\H^0(S,nF)=\H^0(S,\pi^*\O(n))\cong\H^0(\P^1,\O(n))\cong\C^{n+1}.$$
\end{proof}

Let us deform $S$ in the direction $\alpha=(2-n)F+D$ by taking
$$\sigma_t:=\sigma+t\alpha-t^2\left(\frac{\alpha^2}{2\sigma\bar{\sigma}}\right)\bar{\sigma}=\sigma+t((2-n)F+D)+t^2(n-1)\bar{\sigma},$$
where we once again assume the normalization $\sigma\bar{\sigma}=1$. Denote by $\mathcal{S}\rightarrow\Delta$ the resulting deformation of $S$, over a disc $\Delta$ of suitable radius. Then $nF+D$ will remain of type $(1,1)$, because
$$\sigma_t.(nF+D)=\sigma.(nF+D)+t((2-n)F+D).(nF+D)+t^2(n-1)\bar{\sigma}.(nF+D)=0.$$
Moreover, $nF+D$ will be a generator of the N{\'e}ron-Severi group for very general $t$. The linear system $|nF+D|$ gives a relative map
$$\mathcal{S}/\Delta\longrightarrow\P^n_{\Delta}.$$
This is an embedding for very general $t$, but for $t=0$ the image of $S=S_0\rightarrow\P^n$ is a normal rational curve, because $|nF+D|$ has a base locus $D$ on $S=S_0$.

We define a relative moduli space $\mathcal{M}/\Delta$ by defining $M_t$ to be the Mukai moduli space of stable sheaves on $S_t$ with Mukai vector
$$v=(0,[nF+D],1-n).$$
A general element will look like $\iota_*L$ for a degree zero line bundle $L$ on a curve $\iota:C\hookrightarrow S_t$ in $|nF+D|$.

\begin{proposition}
The relative moduli space $\mathcal{M}/\Delta$ is a family of Lagrangian fibrations over $\P^n_{\Delta}$. Specifically
\begin{enumerate}
\item for very general $t$, $M_t$ is a Beauville-Mukai integrable system,
\item for $t=0$, $M_0$ is birational to $\mathrm{Hilb}^nS_0$ and the Lagrangian fibration is induced by the original elliptic fibration on $S_0=S$.
\end{enumerate}
\end{proposition}

\begin{proof}
For very general $t$, $M_t$ is a Beauville-Mukai system by definition; so there is nothing to prove in part 1. When $t=0$, the linear system $|nF+D|$ has $D$ as a base locus. Thus every curve $C$ in the linear system $|nF+D|$ looks like $F_1+\ldots +F_n+D$, where $F_i$ are fibres of the elliptic fibration $S\rightarrow\P^1$. Assume that the $F_i$ are smooth and distinct. A degree zero line bundle $L$ on $F_1+\ldots +F_n+D$ is given by line bundles $L_1,\ldots, L_{n+1}$ on $F_1,\ldots,F_n$, and $D$, respectively, plus isomorphisms $(L_i)_{p_i}\cong (L_{n+1})_{p_i}$ at $p_i=F_i\cap D$, for $1\leq i\leq n$. Stability of $\iota_*L$ implies that the $L_i$ all have degree zero. Therefore the fibre of $M_0\rightarrow |nF+D|$ over the point $F_1+\ldots +F_n+D$ is isomorphic to
$$\mathrm{Pic}^0F_1\times\cdots\times\mathrm{Pic}^0F_n\cong F_1\times\cdots\times F_n.$$

Similarly, the elliptic fibration $S_0\rightarrow\P^1$ induces a Lagrangian fibration
$$\mathrm{Hilb}^nS_0\longrightarrow\mathrm{Sym}^n\P^1\cong\P^n$$
whose fibre over a general point $F_1+\ldots +F_n\in |nF|=\mathrm{Sym}^n\P^1\cong\P^n$ is also isomorphic to $F_1\times\cdots\times F_n$. Since these Lagrangian fibrations both admit global sections, they must be birational.
\end{proof}

As before, a birational modification of the relative moduli space $\mathcal{M}/\Delta$ now yields the following.

\begin{corollary}
The Hilbert scheme $\mathrm{Hilb}^nS_0$ can be deformed to a Beauville-Mukai integrable system, as a Lagrangian fibration over $\P^n$.
\end{corollary}

A Beauville-Mukai integrable system over $\P^n$ will have a discriminant locus of degree $6(n+3)$ (see~\cite{sawon08i}), and general singular fibres will be semistable. The discriminant locus of the Lagrangian fibration on $\mathrm{Hilb}^nS_0$ consists of $24$ hyperplanes and a degree $2n-2$ hypersurface $R$. The latter is the image of the `big' diagonal under the projection
$$\P^1\times\cdots\times\P^1\longrightarrow\mathrm{Sym}^n\P^1\cong\P^n,$$
and its degree can be calculated by identifying $\P^n$ with the space of homogeneous polynomials $f$ of degree $n$; $R$ is then given by the discriminant of $f$, i.e., the resultant of $f$ and $f^{\prime}$, which has degree $2n-2$. The singular fibres above the hyperplanes are semistable, whereas those over $R$ are not. Indeed, the behaviour is similar to the $n=2$ case: over a general point of $R$, the singular fibre will look like the product of $n-2$ elliptic curves and a singular fibre over the conic in the $n=2$ case. Thus we once again see how non-semistable fibres can deform to semistable fibres under a small deformation of the Lagrangian fibration. Moreover, we expect that $R$ should again be counted with multiplicity three, so that the total degree of the discriminant locus,
$$24+3(2n-2)=6(n+3),$$
is preserved under the deformation.

\begin{flushleft}
Department of Mathematics\hfill sawon@email.unc.edu\\
University of North Carolina\hfill sawon.web.unc.edu\\
Chapel Hill NC 27599-3250\\
USA\\
\end{flushleft}

\end{document}